\documentclass[12pt,a4paper,reqno]{amsart}

\usepackage[utf8]{inputenc}
\usepackage[margin=1in,footskip=0.25in]{geometry}
\usepackage{amsmath, amsthm, amssymb}
\usepackage{graphicx}
\usepackage{color}
\usepackage{overpic}
\usepackage{enumerate}
\usepackage{hyperref}

\newcommand{\R}{\mathbb{R}}
\newcommand{\C}{\mathbb{C}}
\newcommand{\N}{\mathbb{N}}

\newtheorem{teo}{Theorem}[section]

\newtheorem{obs}[teo]{Remark}

\newtheorem{lema}[teo]{Lemma}

\newtheorem{defi}[teo]{Definition}
\newtheorem{prop}[teo]{Proposition}

\newcommand{\re}{\mathbb{R}}

\date{}

\newcommand{\sgn}{{\operatorname{sgn}}}

\title[Limit cycles for some families of smooth and non-smooth planar systems]
{Limit cycles for some families of smooth and non-smooth planar
systems}

\author[C. A. Buzzi]{Claudio A. Buzzi}
\address{Mathematics Department, Universidade Estadual Paulista Julio
de Mesquita Filho, 15054-000 S\~ao Jos\'{e} do Rio Preto, S\~ao Paulo,
Brazil } \email{claudio.buzzi@unesp.br}

\author[Y. R. Carvalho]{Yagor Romano Carvalho}
\address{Mathematics Department, Universidade Estadual Paulista Julio
de Mesquita Filho, 15054-000 S\~ao Jos\'{e} do Rio Preto, S\~ao Paulo,
Brazil} \email{yagor.carvalho@unesp.br}

\author[A. Gasull]{Armengol Gasull}
\address{Mathematics Department, Universitat Aut\`{o}noma de Barcelona,
08193 Bellaterra, Barcelona, Catalonia, Spain}
\email{gasull@mat.uab.cat}


\begin{document}

\begin{abstract}
In this paper, we apply the averaging method via Brouwer degree in a
class of planar systems given by a linear center perturbed by a sum
of continuous homogeneous vector fields, to study lower bounds for
their number of limit cycles. Our results can be applied to models
where the smoothness is lost on the set $\Sigma=\{xy=0\}$. We also
apply them to present a variant of Hilbert 16th problem, where the
goal is to bound the number of limit cycles in terms of the number
of monomials of a family of polynomial vector fields, instead of
doing this in terms of their degrees.
\end{abstract}
\maketitle

\section{Introduction }

A limit cycle is a periodic orbit isolated in the set of all periodic orbits in a differential system.
 The existence of limit cycles became important in the applications to the real world, because many
  phenomena are related with their existence, see for instance the Van der Pol oscillator
   \cite{Van1920,Van1926}.  One of the useful tools to detect such objects is the averaging
   theory. We refer to the book of Sanders and Verhulst \cite{SandVerhulst1985} and to the
    book of Verhulst \cite{Verhulst1991} for an introduction of this subject.
    Buica and  Llibre in \cite{BuiLli2004}, generalized the averaging theory for
    studying periodic solutions of continuous differential systems using mainly the Brouwer degree.

The theory of piecewise smooth differential system has been
developing  very fast  and it has become certainly an important
common frontier between Mathematics, Physics and Engineering for
example. In many works on piecewise smooth differential system the
set $\Sigma$, where the systems lose smoothness, is a regular
manifold. But a few years ago it was increasing the study of the
case where $\Sigma$ can be the union of regular manifolds, which
includes, the case when $\Sigma$ is not regular,
 but it is an algebraic manifold.  See for instance Panazzolo and Da Silva in \cite{PanPau2017}.
  Also there are works that deal with the search of limit cycles of discontinuous systems with $\Sigma$
   being an algebraic manifold, see for instance \cite{LliTei2015} and \cite{Nov2014}.

In this work we give some lower bounds for the number of limit
cycles in some classes of continuous, non necessarily locally
Lipschitz, piecewise smooth differential systems with
$\Sigma=\{xy=0\}.$ The main technique will be the averaging theory
via Brouwer degree developed in \cite{BuiLli2004,BuiLliMak2015}.

In Section \ref{se:models} we explain some of the problems that have
motivated our study. They include  systems that model the capillary
rise, some population models and also some type of SIR models. All
of them have in common that can be written as differential equations
of the form
\begin{equation}\label{eq:mot}
\dot x=f(x,y,\sqrt x,\sqrt y), \quad \dot y= g(x,y,\sqrt x,\sqrt y),
\end{equation}
with $f$ and $g$ smooth or polynomial functions. Extending the
function $\sqrt u$ as $\sgn(u)\sqrt{|u|}$ these systems can be
considered in the full plane but they are non-smooth on the set
$\Sigma=\{xy=0\}.$ Clearly, they fall into the category of systems
described above. Notice also than on $\Sigma,$ the corresponding
vector fields are not Lipschitz functions.

In fact these systems could also be treated by introducing new
variables $u$ and $v$ such that $u^2=x$ and $v^2=y$ and changing the
time, but as we will see, our approach can be applied directly to
the original system and can also be applied to more general systems
involving simultaneously more non-differentiable functions. For
instance  functions like $\sqrt[k] x,$ for different values of $k,$
also fall in our point of view.

Recall that a continuous  vector field  $X(x,y)$ is called
homogeneous with degree of homogeneity  $\alpha,$ where $0\le
\alpha\in\R,$ if $X(rx,ry)=r^{\alpha}X(x,y)$ for all $(x,y)\in\R^2$
and all $0\le r\in\R.$ For convenience we will write it as
$X(x,y)=(f(x,y),g(x,y))$ instead of the more usual way
$X(x,y)=f(x,y)\frac{\partial }{\partial
x}+g(x,y)\frac{\partial}{\partial y}.$ When $\alpha<1$ this vector
field is continuous but not Lipschitz. Its associated planar system
of differential equations is $(\dot x,\dot y)=X(x,y),$ or
equivalently, $\dot x=f(x,y),$ $\dot y=g(x,y).$ We prove:

\begin{teo}\label{th:main1}
    Consider the class $\mathcal{F}_{\bf a}$ of planar vector fields
\begin{equation}\label{eq:main}
X(x,y)=(-y,x)+\sum_{j=0}^n a_j X_j(x,y), \quad {\bf
a}=(a_0,a_1,\ldots,a_n)\in\R^{n+1},
\end{equation}
where   for each $j,$  $X_j=(f_j,g_j)$ is a fixed continuous
homogeneous vector field with  degree of homogeneity $0\le
\alpha_j\in\R$ and $\alpha_0<\alpha_1<\cdots<\alpha_n.$ There exist
values of $\bf a$ such that the differential equation associated to
$X$ has at least $m$  limit cycles, where $m+1$ is the number of
non-zero values among
\[
I_j=\int_0^{2\pi} \big(f_j(\cos \theta,\sin \theta) \cos \theta+
g_j(\cos \theta,\sin \theta) \sin \theta \big)\, d\theta, \quad
j=0,1,\ldots,n.
\]

Moreover, if all the vector fields $X_j$ are of class
$\mathcal{C}^1,$ the $m$ limit cycles obtained above are hyperbolic.
\end{teo}

The proof of Theorem~\ref{th:main1} is based on the averaging first
order results for continuous differential equations via Browuer's
degree given in \cite{BuiLli2004,BuiLliMak2015}. This result extends
some of the results of \cite{CimGasMan1995} to the non-smooth case.

Notice that some simple examples of non-smooth $X_j$  where our
approach can be used are for instance
\[
X_j(x,y)= \big(a_j\sgn{(x)}|x|^{\alpha_j} +b_j
\sgn{(y)}|y|^{\alpha_j}, c_j\sgn{(x)}|x|^{\alpha_j} +d_j
\sgn{(y)}|y|^{\alpha_j}\big),
\]
where $0<\alpha_j<1.$ They clearly  include our goal functions.

The second part of the paper deals with polynomial vector fields.
Recall that the second part of the Hilbert's
     16th problem asks about the maximum number of limit cycles for planar polynomial vector fields in terms of
     their degrees. Usually, the maximum number
     of limit cycles of vector fields of degree $n,$ is denoted as $\mathcal{H}(n)$ (admiting, in principle that this number could be infinity)
       and  it is called Hilbert number. To prove its finiteness, and to know it, is one of the most
    famous and difficult open problems in mathematics, see
    \cite{Ily2002,Sma1998}. It is known that  $\mathcal{H}(1)=0,$ $\mathcal{H}(2)\ge4,$
    $\mathcal{H}(3)\ge 13,$ see \cite{ProTor2019} for more lower bounds for small $n$ and other related references.
     It is also known that there is a sequence of values $n$ going
    to infinity such that
    $\mathcal{H}(n)\ge M(n)$ where $M(n)=\big(\frac{n^2\log(n)}{2\log 2}\big)(1+o(1)),$ see for instance \cite{AlvColDePro2020} and their
    references. To the best of our knowledge the first result
    proving the existence of a lower bound of type $O(n^2\log(n))$
    for $\mathcal{H}(n)$ is due to Christopher and Lloyd
    (\cite{ChrLlo1995}).

From the statement of Theorem \ref{th:main1} we start to think into
a different version of Hilbert sixteenth problem facing the question
from a different point of view. Instead of trying to bound the
number of limit cycles in terms of the degrees of the vector fields
we start wondering ourselves if it is not better to do this in terms
of the number of homogeneous vector fields involved in a family.
Very soon, we realize that this leads essentially  to the same
problem, because polynomial vector fields of degree $n$ are the sum
of $n+1$ homogeneous vector fields. In fact, it is even a worst
point of view in the light of the following family of polynomials
vector fields studied in \cite{GasLiTor2015},
\[
\dot z= Az+Bz|z|^{2(k-3)}+{\rm i}c\overline z ^{k-1},
\]
where $z=x+{\rm i}y,$  $A=a_1+{\rm i}a_2,$ $B=b_1+{\rm i}b_2\in\C,$
$c\in\R$ and $k\ge3.$ It has at least $k$ limit cycles but it can
also be written in real variables as
\[
(\dot x,\dot
y)=a_1X_1(x,y)+a_2X_2(x,y)+b_1X_3(x,y)+b_2X_4(x,y)+cX_5(x,y),
\]
that is, involving only five homogeneous vector fields with only 3
different degrees.

Nevertheless this way of thinking the problem lead us to a new point
of view that we hope that results interesting for the reader: Why do
not try to study the number of limit cycles in terms of the number
of homogeneous vector fields formed by  single monomials?

Somehow this point of view tries to mimic the role of Descartes
theorem for studying the number of real zeroes of a polynomial
$P(x)$ of degree $n,$ having  $m$ non-zero monomials. Recall that
while the maximum number of real roots is $n$, the actual maximum
number of real roots is $2m-1$ and this bound is independent of the
degree of $P.$ In fact, $P$ has at most $m-1$ positive roots, $m-1$
negative roots, and eventually the root $0.$

To state more clearly our point of view and our results, for each
$m\in\N$ fixed, we consider the following family  of polynomials
differential equations:

\begin{itemize}
\item Family  $\mathcal{M}_m$ given by
\[
(\dot x,\dot y)=\sum _{j=1}^m a_j X_j(x,y),\quad \mbox{with}\quad
X_j(x,y)=\begin{cases}\big(x^{n_j}y^{k_j},0\big),&  or,\\
\big(0,x^{n_j}y^{k_j}\big),&
\end{cases}
\]
where ${\bf a}\in\R^{m}$ and the couples $(n_j,k_j)\in\N^2$ vary
among all the possible values. Varying $m,$ this family covers all
polynomial differential equations. The letter $\mathcal M$ is chosen
because the important point is to count the number of involved
monomials.
\end{itemize}

We define $\mathcal{H}^M[m]\in\N\cup\{\infty\}$ to be the maximum
number of limit cycles that systems of the family $\mathcal{M}_m$
can have.

Next theorem includes our results about lowers bounds for this
Hilbert type number. The proof of the first part for $m\ge3$ is a
straightforward consequence of Theorem \ref{th:main1} and also a
consequence of other known results about classical Li\'{e}nard systems.
The second part is a direct corollary of the recent paper
\cite{AlvColDePro2020} and uses generalized Li\'{e}nard systems.

\begin{teo}\label{th:main2} With the notation introduced above it holds that $\mathcal{H}^M[m]=0$ for $m=1,2,3$ and for  $m\ge 4,$
$\mathcal{H}^M[m]\ge m-3.$ Moreover, there exits a sequence of
values of $m$ tending to infinity such that $\mathcal{H}^M[m]\ge
N(m),$ where
\[
N(m)=\Big(\frac{(\frac{m-3}2)\log(\frac{m-3}2)}{\log
2}\Big)(1+o(1)).
\]
\end{teo}

A similar result could be stated by using the lower bounds of
$\mathcal{H}(n)$ of type $O(n^2\log(n))$ because the systems of
degree $n$ involve $m=(n+1)(n+2)$ monomials. These systems and the
ones of \cite{AlvColDePro2020} are relevant because for $m$ big
enough they have more limit cycles than monomials.

It is not difficult to see that all results given in Theorem
\ref{th:main2} also hold for the subclass of $\mathcal{M}_m$  of
second order differential equations $\ddot x= P(x,\dot x)$ with $P$
being a polynomial with $m-1$ monomials because classical
Li\'{e}nard differential equations write as $\ddot x= f(x)+g(x)\dot x,$
or equivalently, like the system $(\dot x, \dot y)=(y,-f(x)-g(x)y).$

It is also worth to mention that the celebrated examples of
quadratic systems that prove that $\mathcal{H}(2)\ge4$ are given by
systems with $m=8$ monomials, see for instance
\cite{CheArtLli2003,Per1984}, and so they have $m-4$ limit cycles.
The cubic system given in \cite{LiLiuYan2009} proving that
$\mathcal{H}(3)\ge13$ has $m=9$ monomials and at least $m+4$ limit
cycles.

Under the light of the above results, a  natural problem is to find
the minimal $m$  such that there exists a system with $m$ monomials
having at least $m+1$ limit cycles.

\section{Some motivating models}\label{se:models}

In this section we shortly explain some models that motivate the
class of equations \eqref{eq:mot} that can be treated with the tools
introduced  in this paper.

\subsection{Capillary rise}\label{ss:cap} A first example is given by the equation that
models the capillary rise.  The capillary action is a physical
  property that the fluids have in to go down or up in extremely thin tubes.
  Sometimes this action to do the liquid to go up against the force of gravity
  or even to induce a magnetic field.  This ability to rise or fall results from
   the ability of the liquid to ``wet'' or not the pipe surface (glass, plastic,
    metal, etc.).  For instance  in the case of water in a glass beaker, we have
     tendency of water to adhere to the glass, bending upward near the wall, forming
      a concave meniscus and rising to a certain height above water level, here we
       have a capillary rise.  In the case of mercury the opposite happens,
        the tendency of mercury is to move away from the wall, forming a convex
         meniscus and descending at a certain height from the mercury level,
          here we have a capillary depression.

 This phenomenon is described in more detail in \cite{PloSwi2018}
   and can be modeled in an adimensional way by the planar system
\begin{equation*}
\left\{ \begin{array}{l}
\dot x=y \\
\dot y = 1 - a y - \sqrt{2x}
\end{array} \right.,
\end{equation*}
where  $a$ is a positive parameter.

\subsection{Some population models}

Following \cite{AjrPitVen2011} we introduce the herd behavior. If
$R$ represents the density of certain population, namely number of
individuals per surface unit, with the herd occupying an area $A,$
then the individuals who take the outermost positions in the herd
are proportional to the perimeter of the region where the herd is
located whose length depends on $\sqrt{A}.$ They are therefore in
number proportional to the square root of the density, that is to
$\sqrt{R},$ with a proportionality constant that depend on the shape
of the herd.  Then, the interactions with the second population with
density $Q$ occur only via these peripheral individuals, so that
instead of the standard  $RQ$ that appears in the usual
predator-prey systems, there is a term proportional to $\sqrt{R}Q.$
In a dimensional-less set of variables these type of models write as
\begin{equation*}
\left\{ \begin{array}{l}
\dot x=x(1-x)-y\sqrt x, \\
\dot y = -xy+c y\sqrt{x},
\end{array}\right.
\end{equation*}
for $c\in\R,$ see also \cite{Bra2012}. For other population models,
involving also square roots, see \cite[Sec. 4.10]{Braun}.

\subsection{ A SIR type model} In \cite{Mic2012} the author
proposes a variation of the classical SIR model.  Recall that it is
a mathematical model of the spread of infectious diseases that
classifies the population in three categories: Susceptible,
Infectious, or Recovered. This model relate these categories by the
differential system
\begin{equation*} \dot S= -\beta
\sqrt{SI},\quad \dot I =\beta\sqrt{SI}-\gamma\sqrt{I},\quad \dot
R=\gamma\sqrt I
\end{equation*}
where $\alpha,\beta$ and $\gamma$ are real parameters. Notice that
it can be studied via a planar system because $\dot S+\dot I+\dot
R=0$ and as a consequence $S(t)+I(t)+R(t)=S_0+I_0+R_0.$

\section{Definitions and Preliminaries. } 

In this section we review some definitions and results that will be used in this paper. For the
characterization of Chebyshev Systems in an open interval we will use the following results which
 can be found in \cite{KarStu1966} and \cite{Mardesic1998}.

\begin{defi}
    Let $ u_0,\ldots, u_ {n-1}, u_n $ be functions defined in an open interval $ L $ of $ \re $.
    The ordered set
     $ (u_i )_ {i = 0} ^ n $ forms an extended complete Chebyshev system, for short $
         ECT $-system, on $ L $ if any nontrivial linear combination $ a_0u_0 +\cdots+ a_k u_k $
          has at most $ k $ isolated roots in $ L $ counting multiplicity, for every $ k = 0,1,\ldots, n.$
\end{defi}

The following result is a very useful characterization of smooth
$ECT$-systems in terms of Wronskians.
\begin{prop}\label{lemasistemaTintervaloaberto}
    The set of ordered $\mathcal{C}^n$-functions $(u_0, \ldots, u_n)$ forms an $ ECT $-system on $ L $  if, and only if,
     for every $ k = 0, ..., n $,
    \begin{equation*}
    W(u_0,\ldots,u_k)(x) =\left| \begin{array}{ccc}
    u_0(x) & \cdots & u_k(x) \\
    u_0'(x) &  \cdots & u_k'(x) \\
    \vdots &  \ddots & \vdots \\
    u_0^{(k)}(x)  & \cdots & u_k^{(k)}(x)
    \end{array} \right|\ne0,
    \end{equation*}
 for every $ x \in L $.
\end{prop}

We will need the following lemma.

\begin{lema}\label{lemawronsk}
    Consider $\beta_i \in \re$ such that $\beta_0<\beta_1<\cdots<\beta_m$. Then the functions
     $(x^{\beta_0},\dots,x^{\beta_m})$ form an $ECT$-system on $(0,\infty).$
\end{lema}

\begin{proof} We claim that
\begin{equation}\label{formula}
    \begin{aligned}
    W=W(x^{\beta_0},\ldots,x^{\beta_k})=x^S\Big(\displaystyle\prod_{0\le i<j\le k}^k(\beta_j-\beta_i)\Big), \quad \mbox{where}\quad
    S=\sum_{i=0}^k \beta_i-
    \frac{k(k+1)}2.
    \end{aligned}
    \end{equation}
Then,  each $W(x^{\beta_0},\dots,x^{\beta_{k}}) \neq0$ in
$(0,\infty)$, for     $k=0,\dots,m$,
      and by Proposition \ref{lemasistemaTintervaloaberto} the functions
      $\big(x^{\beta_j}\big)_{j=0}^m$ form an $ECT$  on $(0,\infty)$
      as we wanted to prove.

Let us prove the claim. For $1\le k\in\N,$ set
$(\beta)_k=\beta(\beta-1)(\beta-2)\cdots (\beta-k).$ Then,
\begin{align*}
W=&\left| \begin{array}{ccc}
    x^{\beta_0} & \cdots & x^{\beta_k} \\
    \beta_0  x^{\beta_0-1} &  \cdots & \beta_k  x^{\beta_k-1} \\
    (\beta_0)_1  x^{\beta_0-2} &  \cdots & (\beta_k)_1
    x^{\beta_k-2}\\
    \vdots &  \ddots & \vdots \\
(\beta_0)_{k-1}  x^{\beta_0-k} &  \cdots & (\beta_{k})_{k-1}
x^{\beta_k-k}
    \end{array} \right|=x^S \left| \begin{array}{ccc}
    1 & \cdots & 1 \\
    \beta_0  &  \cdots & \beta_k   \\
    \beta_0(\beta_0-1)  &  \cdots & \beta_k(\beta_k-1)\\
    \vdots &  \ddots & \vdots \\
(\beta_0)_{k-1}  &  \cdots & (\beta_k)_{k-1}
    \end{array} \right|\\=&x^S \left| \begin{array}{ccc}
    1 & \cdots & 1 \\
    \beta_0  &  \cdots & \beta_k   \\
    \beta_0^2  &  \cdots & \beta_k^2\\
    \beta_0(\beta_0-1)(\beta_0-2)  &  \cdots & \beta_k(\beta_k-1)(\beta_k-2)\\
    \vdots &  \ddots & \vdots \\
(\beta_0)_{k-1}  &  \cdots & (\beta_k)_{k-1}
    \end{array} \right|=\cdots=x^S \left| \begin{array}{ccc}
    1 & \cdots & 1 \\
    \beta_0  &  \cdots & \beta_k   \\
    \beta_0^2  &  \cdots & \beta_k^2\\
    \beta_0^3  &  \cdots & \beta_k^3\\
    \vdots &  \ddots & \vdots \\
\beta_0^k  &  \cdots & \beta_k^k
    \end{array} \right|,
\end{align*}
where this last determinant is the celebrated Vandermonde
determinant and coincides with expression \eqref{formula}. Notice
that in the first equality we have used that taking the product of
$k+1$ elements of the determinant, being each one of them  elements
of different rows and columns, always appears $x^S$ as a factor.
Moreover, in the first equality of the second line of equalities we
have changed the third file by the sum of the second and third files
of the previous determinant. Similarly, we change the fourth file of
this new determinant by a suitable linear combinations of the
second, third and fourth ones, and so on, until arriving to the
final equality. So, the claim follows.
\end{proof}

A second key tool for proving Theorem \ref{th:main1} will be next
averaging type Theorem, proved in \cite{BuiLli2004}, which is
applicable to continuous differential systems. See the Appendix for
a short reminder about Brouwer topological degree.

\begin{teo}\label{teoaveraging}
   {\rm (Averaging theorem via Brouwer degree  (\cite{BuiLli2004})).} Consider the system of differential equations
    \begin{equation}\label{sistemaformapadraoaveraging}
    x'(t) = \varepsilon H(t,x) + \varepsilon^2 K(t,x,\varepsilon),
    \end{equation}
    where $H: \re \times D \rightarrow \re^n, \; K: \re \times D \times (-\varepsilon_0,
    \varepsilon_0)\rightarrow \re^n$ are continuous functions, $ T $-periodic in the first variable,
    $ D $ is an open subset of $ \re ^ n $ and $ (- \varepsilon_0,
\varepsilon_0) $ is a neighborhood of $0\in\re$.  We define the
averaged function, $h: D \rightarrow \re^n$ as follow:
    \begin{equation}\label{funcaomediaaveraging}
    h(z) = \dfrac{1}{T} \displaystyle\int_0^T H(t,z)\,dt,
    \end{equation}
    and we assume that each $a\in D$ with $h(a) = 0$, there is a neighborhood $V$ of $a$ such that
     $h(z) \neq 0$ for all $z  \in \overline{V} \setminus \{a\}$ and Brouwer degree $d_B(h,V,0) \neq 0$. Then  for
      each $|\varepsilon| >0$ small enough, there is a $T$-periodic solution $\varphi(\cdot, \varepsilon)$
      of system (\ref{sistemaformapadraoaveraging}) such that $\varphi(\cdot, \varepsilon) \rightarrow a$  when $\varepsilon \rightarrow 0$.
\end{teo}

\begin{obs}
    Theorem \ref{teoaveraging} shows that for each isolated solution $a$ of $h(z)=0$ in $D,$
    where $h$ is given in \eqref{funcaomediaaveraging}, such that $d_B(f,V,0) \neq 0$,
there is, for $\varepsilon$ small enough,   a $T$-periodic orbit of
system \eqref{sistemaformapadraoaveraging} tending to $a$ when
$\varepsilon$ goes to $0.$ When $h$ is of class $\mathcal{C}^1$
these hypotheses about  the solution $a$ can simply be replaced by
$h'(a)\ne0.$
\end{obs}

The key result for proving the second part  of Theorem
\ref{th:main2} will be the following theorem.

\begin{teo}\label{th:lienard} {\rm (\cite{AlvColDePro2020})} There is a sequence of naturals numbers $n,$ tending to
infinity, such that for these values of $n$ there exist  generalized
Li\'{e}nard systems
\begin{equation*}
\left\{ \begin{array}{l}
\dot x=y-F(x), \\
\dot y = G(x),
\end{array} \right.
\end{equation*}
with $F$ and $G$ are polynomials of degree at most $n,$ having at
least $K(n)$ limit cycles, where
\[
K(n)=\Big(\frac {n\log n}{\log 2}\Big)(1+o(1)).
\]
\end{teo}

We also will need the following result about non-existence of limit
cycles.

\begin{prop}\label{le:lema} Systems
\[
(\dot x,\dot y)= \big(ax^py^q, bx^iy^j+cx^ky^l\big),
\]
where $(a,b,c)\in\R^3$ and $(p,q,i,j,k,l)\in\N_0^6,$ with
$\N_0=\N\cup\{0\},$ have no limit cycle.
\end{prop}

\begin{proof} We will use the following well-known properties for proving
non-existence of limit cycles:

\begin{itemize}
\item [$P_1$:] Periodic orbits must surround some critical point. So
systems without critical points have no periodic orbit.

\item[$P_2$:] If a system has an invariant line passing by all its
critical points, if any, then it has no periodic orbits. This is so
by property $P_1$ if the system has no critical points, or,
otherwise, by uniqueness of solutions, because an eventual periodic
orbit would surround some of the critical points and as a
consequence, cut the line.

\item[$P_3$:] If one of the two differential equations only involves ones of
the variables (for instance $\dot x=f(x)$) then the system has no
periodic orbits. This is so, because autonomous one dimensional
ordinary differential equations have no non-constant periodic
solution.

\item[$P_4$:] If a planar system has a smooth
first integral defined on an open set $\mathcal{U}\subset\R^2,$
although it can have continua of periodic orbits, it can not have
limit cycles entirely contained in  $\mathcal{U}.$

\item[$P_5$:] If the divergence of a planar system $(\dot x,\dot
y)=(P(x,y),Q(x,y))$, $\operatorname{div} (P,Q)= \frac{\partial
P(x,y)}{\partial x}+\frac{\partial Q(x,y)}{\partial y}$ does not
change sign and vanishes only on sets of zero Lebesgue measure, then
the system does not have periodic orbits.

\item[$P_6$:] Let $X$ be a planar vector field with a unique
critical point, $(0,0),$ and assume that it is reversible, that is,
invariant by one of the two changes of variables and time:
\[
(x,y,t)\longrightarrow (-x,y,-t)\quad \mbox{or} \quad
(x,y,t)\longrightarrow (x,-y,-t).
\]
If the system has a periodic orbit that crosses transversally the
axes then it is in the interior  of a continua of periodic orbits
and it is not a limit cycle. This is so, because any of  the
described symmetries  implies that if an orbit turns around the
origin it is periodic. Sometimes this criterion is called {\it
reversibility criterion of Poincar\'{e}}, because he was the first in
using it for proving the existence of periodic orbits.
\end{itemize}

When $a=0,$  the system can not have periodic orbits because of
property $P_3.$ When $bc=0,$ we assume, for instance, that $c=0$ and
$b\ne0,$ because when $c\ne0$ and $b=0$ the situation is the same
and the case $b=c=0$ is trivial. Then, two situations may happen:
either $x=0$ or $y=0$ are a continuum of critical points and no
other critical points appear or it writes as $(\dot x,\dot y)=
(ay^q, b x^i).$ In the first case the only critical points belong to
an invariant line full of critical points, so the system can not
have periodic orbits by property $P_2.$ In the second case the
system is integrable  with $\mathcal{U}=\R^2$ and by property $P_4$
no limit cycle appears.

Hence, from now on, we will assume that $abc\ne0.$ Next step will
use that the phase portraits of any two systems of the form
\[
(\dot x,\dot y)=(P(x,y)R(x,y),Q(x,y)R(x,y))\quad\mbox{and}\quad
(\dot x,\dot y)=(P(x,y),Q(x,y)),
\]
are the same (modulus a change of time orientation) in each
connected component of $\R^2\setminus\{R(x,y)=0\}.$ We will take $R$
as some suitable polynomial of one of the forms $x^n$  or $y^n$ to
reduce our  study to simpler vector fields. Taking $R(x,y)=x^s$ for
$s$ to be the minimum of $p,i$ and $k$ we can reduce the situation
to one of the next two differential systems:
\begin{equation}\label{eq:2cases}
(\dot x,\dot y)= \big(ay^q, bx^iy^j+cx^ky^l\big),\, i\ge 1
\quad\mbox{or}\quad (\dot x,\dot y)= \big(ax^py^q,
by^j+cx^ky^l\big),
\end{equation}
where for simplicity we keep the same notation for the new exponents
and without loss of generality we have assumed that $i\le k.$ Next
we take $R(x,y)=y^u$ with $u$ being the minimum of $q,j$ and $l.$
Finally, we only need to study the following five cases:
\begin{align*}
&(i) \quad (\dot x,\dot y)= \big(a, bx^iy^j+cx^ky^l\big),\, i\ge1,
&&(ii) \quad (\dot
x,\dot y)= \big(ay^q, bx^i+cx^ky^l\big),\, i\ge 1,\, q\ge 1,\\
&(iii) \quad(\dot x,\dot y)= \big(ax^p, by^j+cx^ky^l\big),
&&(iv)\quad (\dot x,\dot y)= \big(ax^py^q, b+cx^ky^l\big),\\
&(v) \quad (\dot x,\dot y)= \big(ax^py^q, by^j+cx^k\big), &&
\end{align*}
where we also keep the old notation for the new exponents. Notice
that $(i)$ and $(ii)$ come from the first differential equations of
\eqref{eq:2cases} and the other three cases from the second one.

 The case
$(i)$ has no critical point, so it has no periodic orbit by property
$P_1.$

In case $(ii),$ when  $l=0$ we can apply property $P_4$ with
$\mathcal{U}=\R^2$ because the system has a polynomial first
integral.

When   $l\ne0$ the system has a unique critical point $(0,0)$ and it
writes as
\begin{equation}\label{eq:c2}
(\dot x,\dot y)=(ay^q,bx^i+c x^ky^l),\quad i\ge1,\, q\ge1,\, l\ge 1.
\end{equation}

 Notice that studying the vector field on the axes we get
\[
\dot x\big|_{x=0}= a y^q \quad\mbox{and}\quad \dot y\big|_{y=0}=b
x^i.
\]
Since a periodic orbit must surround the origin, the above
conditions imply that this is only possible when $q$ and $i$ are
both odd numbers and $ab<0.$ So, in this case we will assume that
these conditions hold because otherwise the system has not periodic
orbits.

If $l$ is even the system is invariant by the change
$(x,y,t)\longrightarrow (x,-y,-t)$ and by property $P_6$ the system
has no limit cycle and  we are done. If $k$ is odd, then the system
is invariant by the change $(x,y,t)\longrightarrow (-x,y,-t)$ and
again by property $P_6$ we are done. Hence it only remains to
consider the case $l$ odd and $k$ even. Notice that
\[
\operatorname{div}(X)= c l x^k y^{l-1},
\]
and then it does not change sign and only vanishes on $\{xy=0\},$ or
on one subset of $\{xy=0\}.$ Hence by property $P_5$ the system has
no periodic orbit.

In case $(iii),$  we use property $P_3$.

In case $(iv)$ when $pqkl\ne0$ we can apply property $P_1.$ Also,
when $p=q=0$ we can apply property $P_1.$ Next we split the study
according one of the variables $p,q,k$ or $l$ vanishes and taking
into account that  $p^2+q^2\ne0.$

Assume that $p=0.$ Then $q\ne0.$ When $l\ne0$ we can apply again
property $P_1.$ When $l=0$ we can apply property $P_4$  because the
system has a polynomial first integral.

Assume that $q=0.$ Then the first equation of the system is $\dot
x=ax^p$ and we can apply property $P_3.$

Assume that $k=0.$ Then the second equation of the system is $\dot
y=b+cy^l$ and we can apply again property $P_3.$

Finally, assume that $l=0.$ When $p=0$ the system has a polynomial
first integral and we can apply property $P_4$ with
$\mathcal{U}=\R^2.$ When $p\ne0,$ the system has the invariant line
$\mathcal{L}=\{x=0\},$ and it can be integrated by separating the
variables, giving an smooth first integral in $\R^2\setminus
{\mathcal L}.$ Then we can apply again property $P_4$ to each of the
connected components of $\R^2\setminus {\mathcal L}$ and prove the
non-existence of limit cycles because $\mathcal L$ is also invariant
and eventual limit cycles can not cut it.

Finally we study case $(v)$. When $q=0$ by property $P_3$ no
periodic orbit appears. We consider four diferent subcases that
cover all the situations.

When $q\ne0,$ $p=0$ and $j=0$ the system has a polynomial first
integral and by property $P_4$ we are done.

When $q\ne0,$ $p=0$ and $j\ne0$ the system writes as
\begin{equation}\label{eqcaso5}
(\dot x,\dot y)= (ay^q, by^j+c x^k),\quad q\ge1,\, j\ge 1.
\end{equation}
If $k=0$ in \eqref{eqcaso5} then  we use property $P_3$.  The case $k \neq 0$ in \eqref{eqcaso5} we notice that,
by changing the names of same of the parameters it
coincides with the system \eqref{eq:c2} studied in case $(ii)$
taking in that system $k=0.$ Hence, again this system has no limit
cycle.

When $q\ne0,$ $p\ne0$ and $j=0$ the system has once more the
invariant line $\mathcal{L}=\{x=0\},$ and it can be integrated by
separating the variables, giving an smooth first integral in
$\R^2\setminus {\mathcal L}.$ As in the similar previous situation,
we can prove that it has no limit cycles by using property $P_4.$

In the remaining  case $q\ne0,$ $p\ne0$ and $j\ne0.$ Then the
$(0,0)$ is its unique critical point and $x=0$ is an invariant line.
By property $P_2$ it has no periodic orbit.

Hence we have proved that although sometimes the system has continua
of periodic orbits it has not limit cycles, as is stated in the
lemma.
\end{proof}

\section{Proof of Theorem 1.1}

     In
    order to find periodic orbits for the continuous planar
    differential system
    \[
(\dot x,\dot y)=  (-y,x)+\varepsilon \sum_{j=0}^n b_j X_j(x,y),
    \]
we will apply
    the averaging method via Brouwer degree given by the Theorem
    \ref{teoaveraging}.
Notice that we haven taken ${\bf a}=\varepsilon {\bf b}$ in the
expression \eqref{eq:main} and $\varepsilon$ is a small parameter.
As usual, we write the system in polar coordinates $x=r\cos \theta,$
$y=r\sin\theta,$ see for instance \cite{BuiLli2004}. We get
\begin{align*}
\dot r =& \varepsilon \sum_{j=0}^n b_j \big(x f_j(x,y)+ y
g_j(x,y)\big)=\varepsilon\sum_{j=0}^n b_j F_j(\theta)r^{\alpha_j},
\\
\dot \theta =& 1+ \varepsilon \sum_{j=0}^n b_j \big(x g_j(x,y)-y
f_j(x,y)\big)=1+ \varepsilon \sum_{j=0}^n b_j G_j(\theta)
r^{\alpha_j-1},
\end{align*}
where
\begin{align*}
F_j(\theta)=&f_j(\cos \theta,\sin \theta)\cos \theta + g_j(\cos
\theta,\sin \theta)\sin \theta,\\ G_j(\theta)=& g_j(\cos \theta,\sin
\theta)\cos \theta- f_j(\cos \theta,\sin \theta)\sin\theta.
\end{align*}
Finally, we have the differential equation
\begin{equation}\label{eq:pol}
\frac{dr}{d\theta}=r'=\frac{\varepsilon\sum_{j=0}^n b_j
F_j(\theta)r^{\alpha_j}}{1+ \varepsilon \sum_{j=0}^n b_j G_j(\theta)
r^{\alpha_j-1}}=\varepsilon\sum_{j=0}^n b_j
F_j(\theta)r^{\alpha_j}+O(\varepsilon^2).
\end{equation}
It is continuous for $(\theta,r)\in\R\times(0,R_0)$ for some $R_0>0$
and $\varepsilon$ small enough. We can easily compute the averaged
function $h$ given in Theorem \ref{teoaveraging}. We obtain
\[
h(z)= \frac1{2\pi}\int_0^{2\pi} \sum_{j=0}^n b_j
F_j(\theta)z^{\alpha_j}\, d\theta= \sum_{j=0}^n \frac{b_j}{2\pi}
\Big(\int_0^{2\pi} F_j(\theta)\, d\theta\Big)z^{\alpha_j}=
\sum_{j=0}^n \frac{b_j I_j}{2\pi} z^{\alpha_j}
\]
Since, from all $I_j,j=0,1,\ldots,n,$ only $m+1$ values are
non-zero, we rename the corresponding ordered $\alpha_j$ as
$\beta_0,\beta_1,\ldots,\beta_m$ and then
\[
h(z)=\sum_{j=0}^m c_j z^{\beta_j},
\]
with all $c_j$ arbitrary real constants  and
$\beta_0<\beta_1<\cdots<\beta_m.$ By Lemma \ref{lemawronsk} they
form an $ECT$-system on $(0,\infty).$ In particular, the maximum
number of  positive zeroes of $h$ is $m$ and there exist
$c_0,c_1,\ldots,c_m$ such that $h$ has exactly $m$ simple zeroes.
Notice that the upper bound of $m$ zeroes for $h$ is also a
straightforward consequence of Descarte's rule of signs. Taking the
corresponding values of ${\bf b}$ we obtain a system with
$|\varepsilon|$ small enough and at least $m$ periodic orbits. In
general we do not know yet that these periodic orbits are limit
cycles, that is, isolated among all the existing periodic orbits.
Nevertheless, because the right hand side of our differential equation \eqref{eq:pol}
 is continuous with respect to $\theta$
and diferenciable with respect to $r>0$ we can apply Theorem 3 of
\cite{BuiLliMak2015} that asserts that the obtained periodic orbits
are indeed limit cycles.

To prove the hyperbolicity in the smooth case it suffices to show
that in this regular setting the positive zeroes of the averaged
function $h$ coincide with the ones of the first order Melnikov
function. In \cite{BuiLli2004} this fact is proved in several
situations. We show this result again  for the $\mathcal{C}^1$
planar differential equations of the form
\begin{equation*}
(\dot x,\dot y)=(-y,x)+\varepsilon \big(P(x,y),Q(x,y)\big).
\end{equation*}
Recall that for this system, the Melnikov function $M$ writes as
\[
M(k)=\int_{x^2+y^2=k} P(x,y)\,dy-Q(x,y)\,dx, \quad 0<k\in\R,
\]
see for instance \cite{ChrLi2007}. By parameterizing the circles as
$x=\sqrt{k}\cos \theta,$ $y=\sqrt{k}\sin\theta$ we get that
$M(k)=\sqrt{k}h(\sqrt{k})$ and, as a consequence, the positive
simple zeroes of $h$ give rise to hyperbolic limit cycles of our
planar system for $|\varepsilon|$ small enough. Hence, the theorem
is proved.


\subsection{Examples of application}

As a first application we prove that the simple differential system
\begin{equation}\label{eq:lin}
\left( \begin{array}{c}
         \dot x\\
         \dot y
       \end{array}
\right)=    \left( \begin{array}{c}
         s_1\\
         s_2
       \end{array}
\right)  + \left( \begin{array}{cc}
                 q_{1,1} & q_{1,2} \\
                 q_{2,1} & q_{2,2}
               \end{array}
\right)  \left( \begin{array}{c}
         \surd(x)\\
         \surd(y)
       \end{array}
\right)  +   \left( \begin{array}{cc}
                 p_{1,1} & p_{1,2} \\
                 p_{2,1} & p_{2,2}
               \end{array}
\right)  \left( \begin{array}{c}
         x\\
         y
       \end{array}
\right)  ,
\end{equation}
where $\surd(z)=\sgn(z)\sqrt{|z|},$ has  for some values of the
parameters a limit cycle crossing $\Sigma=\{xy=0\}.$ This family
includes for instance the one given in Subsection \ref{ss:cap}.

In the  notation of the theorem, all systems of the form
\eqref{eq:lin} can be written as
\[
(\dot x,\dot y)= \sum _{j=0}^2 a_j X_j(x,y),
\]
where $X_0(x,y)=(s_1,s_2),$
$X_1(x,y)=(q_{1,1}\surd(x)+q_{1,2}\surd(y),q_{2,1}\surd(x)+q_{2,2}\surd(y))$
and $X_2(x,y)=(p_{1,1}x+p_{1,2}y,p_{2,1}x+p_{2,2}y).$ Moreover
$(\alpha_0,\alpha_1,\alpha_2)=(0,1/2,1).$  Notice that for
simplicity we keep the same names for the constants although they
have varied. Clearly,
\begin{align*}
I_0=&\int_0^{2\pi} (s_1\cos\theta+ s_2\sin\theta)\,d\theta=0,\\
I_2=&\int_0^{2\pi} (p_{1,1}\cos^2\theta+
(p_{1,2}+p_{2,1})\sin\theta\cos\theta+
p_{2,2}\sin^2\theta)\,d\theta=(p_{1,1}+ p_{2,2})\pi,\\
I_1=&4(q_{1,1}+q_{2,2})\int_0^{\pi/2} \cos^{3/2}\theta\,d\theta,
\end{align*}
where in the last equality we have used that
\begin{align*}
\displaystyle\int_0^{2 \pi} \surd(\cos \theta) \cos \theta \, d
\theta=&2\displaystyle\int_{-\frac{\pi}{2}}^{\frac{\pi}{2}}
\cos^{3/2}\theta\, d \theta=4 \displaystyle\int_0^{\frac{\pi}{2}}
\cos^{3/2}\theta \, d \theta>0,\\
\displaystyle\int_0^{2 \pi} \surd(\sin \theta) \sin \theta \, d
\theta=&4 \displaystyle\int_0^{\frac{\pi}{2}} \sin^{3/2}\theta \, d
\theta=4 \displaystyle\int_0^{\frac{\pi}{2}} \cos^{3/2}\theta \, d
\theta,
\end{align*}
and by symmetry,
\[
\displaystyle\int_0^{2 \pi} \surd(\sin \theta) \cos \theta \, d
\theta= \displaystyle\int_0^{2 \pi} \surd(\cos \theta) \sin \theta
\, d \theta=0.
\]

Thus when $(p_{1,1}+ p_{2,2})(q_{1,1}+ q_{2,2})\ne0,$ the number of
non-zero values in the list $I_0,I_1,I_2$ is $2$ and by Theorem
\ref{th:main1} we have a system of the form \eqref{eq:lin} with 1
limit cycle.

As a second example of application consider
\begin{equation}\label{eq:lin2}
\left( \begin{array}{c}
         \dot x\\
         \dot y
       \end{array}
\right)=    \left( \begin{array}{c}
         s_1\\
         s_2
       \end{array}
\right)  +   \left( \begin{array}{cc}
                 p_{1,1} & p_{1,2} \\
                 p_{2,1} & p_{2,2}
               \end{array}
\right)  \left( \begin{array}{c}
         x\\
         y
       \end{array}
\right)   + \left( \begin{array}{cc}
                 q_{1,1} & q_{1,2} \\
                 q_{2,1} & q_{2,2}
               \end{array}
\right)  \left( \begin{array}{c}
         \sqrt[3] x\\
         \surd(y)
       \end{array}
\right),
\end{equation}
where recall that $\surd(y)=\sgn(y)\sqrt{|y|}.$ We will prove that
it has at least 2 limit cycles crossing $\Sigma=\{xy=0\}$ for same
values of the parameters.

Writing it in the  notation of Theorem \ref{th:main1} we get
\[
(\dot x,\dot y)= \sum _{j=0}^3 a_j X_j(x,y),
\]
where $X_0(x,y)=(s_1,s_2),$
$X_1(x,y)=(q_{1,1}\sqrt[3]{x},q_{2,1}\sqrt[3]{x}),$
$X_2(x,y)=(q_{1,2}\surd(y),q_{2,2}\surd(y))$ and
$X_3(x,y)=(p_{1,1}x+p_{1,2}y,p_{2,1}x+p_{2,2}y).$ Moreover,
$(\alpha_0,\alpha_1,\alpha_2,\alpha_3)$ $=(0,1/3,1/2,1).$ Notice
that again, for simplicity, we keep the same names for the constants
although they have varied. In this case,
\[
I_0=0,\quad I_1= 4 q_{1,1}\displaystyle\int_0^{\frac{\pi}{2}}
\cos^{4/3}\theta \; d \theta, \quad I_2=4q_{2,2}\int_0^{\frac{\pi}{2}}
\sin^{3/2}\theta\,d\theta,\quad I_3=(p_{1,1}+ p_{2,2})\pi.
\]
where $I_0,I_2$ and $I_3$ are obtained similarly that in the
previous case  and to get $I_1$ we have used that
\begin{equation*}
\displaystyle\int_0^{2 \pi} \sqrt[3]{ \cos \theta} \cos \theta \; d
\theta=4 \displaystyle\int_0^{\frac{\pi}{2}}  \cos^{4/3}\theta \; d
\theta>0\quad\mbox{and}\quad \int_0^{2 \pi} \sqrt[3]{ \cos \theta}
\sin \theta \; d \theta=0.
\end{equation*}
Hence, when $q_{1,1}q_{2,2}(p_{1,1}+ p_{2,2})\ne0,$  the number of
non-zero values in the list $I_0,I_1,I_2,I_3$ is $3$ and by Theorem
\ref{th:main1} we have an example of system \eqref{eq:lin2} with at
least 2 limit cycles.

\section{Proof of Theorem \ref{th:main2}}
That $\mathcal{H}^M[j]=0,$ for $j=1,2,3,$ is a straightforward
consequence of Proposition \ref{le:lema}. Notice that this
proposition covers all cases except the trivial ones, where  either
$\dot x=0$ or $\dot y=0,$ and the right-hand side of the other
equation has $j$ monomials.

Let us prove that for $m\ge 4,$ $\mathcal{H}^M[m]\ge m-3.$ Consider
the Li\'{e}nard classic system in class $\mathcal{M}_m,$
\[
(\dot x,\dot y)=(y, -x+ a_0y+a_1y^3+\cdots +a_{m-3}y^{2m-5}).
\]
With the notation of Theorem \ref{th:main1} we get that for all
$j=0,1,\ldots,m-3,$
\[
I_j=\int_0^{2\pi} \sin^{2j+2}\theta\,d\theta>0,
\]
and as a consequence we get examples with $m-3$ limit cycles. In
fact, this system includes the celebrated van der Pol system when
$m=4$ and coincides with the example of classical Li\'{e}nard  system
studied in \cite{Lin1980}, where the author, with another notation,
already proved the existence of $m-3$ limit cycles.

Notice that there are many different families in $\mathcal{M}_m$
with at least $m-3$ limit cycles. For instance it suffices to
consider systems of the form
\[
(\dot x,\dot y)=(y, -x+
a_0x^{2n_0}y^{2k_0+1}+a_1x^{2n_1}y^{2k_1+1}+\cdots +
a_{m-3}x^{2n_{m-3}}y^{2k_{m-3}+1}),
\]
with  $n_j,k_j\in\N_0$ and all $2(n_j+k_j),$ $j=0,1,\ldots m-3,$
taking different values. Also similar terms could be added in the
first differential equation, removing some other ones from the 
second one.

To prove that $\mathcal{H}^M[m]\ge N(m)$ we  will use Theorem
\ref{th:lienard}. For a sequence of values of $n$ tending to
infinity, the number of monomials of these generalized Li\'{e}nard
systems is $m=2n+3$ while their number of limit cycles is at least
$K(n).$  Hence  these systems are in $\mathcal{M}_m$ and have at
least $N(m)=K((m-3)/2)$ limit cycles. This function is the one that
appears in the statement of the theorem.

\section*{Appendix: The Brouwer degree}

 In this appendix we include  a short introduction about  the Brouwer degree.   We will  simply present some
  definitions and aspects about it. For more details we recommend \cite{OutRui2009}.
   From now on, $V$ stands for a bounded open set in  $\re^n$, and $\partial V$ is the boundary of the $V$ set.
     Our purpose here is to define the degree of a continuous mapping $f:\overline{V} \rightarrow \re^n$.
      Thus we will do by the usual method: first we consider smooth mappings, and then we extend the
       definition to any given continuous mapping.

\begin{defi}
Consider $ V $ a non-empty, open and limited subset of $ \re^ n $.
We define $ \mathcal{C} ^ k (\overline{V}, \re^n) $
 as the space of the  $ k $-times continuously differentiable functions into $ \overline {V} $.  If the function
  is continuously differentiable for any $k$ then we said that the function is in  $\mathcal{C}^{\infty}(\overline{V},\re^n)$ or it is smooth.
\end{defi}

Note that, $\overline{V}$ being compact and $f: \overline{V}
\rightarrow \re^n$  a continuous mapping
 then we can define the norm
$$||f||=\max\{||f(x)|| \; ; \; x \in \overline{V}\}.$$

As mentioned above, we will consider the case when f is smooth.  In
particular we will be dealing with the set $R_{f|_{_V}} \subset
\re^n$ of regular values of $f$ in $V$.  Recall that $b \in
R_{f|_{_V}}$
 when the derivative $Df(x)$ is bijective at every point $x \in V$ with $f(x)=b$ (trivially true if $b \notin f(V)$),
  that is, $R_{f|_{_V}}=f(R)$ where $R = \{ x \in V \; ; \; \det(Df(x)) \neq 0 \}$ and $\det (Df(x))$
  is the determinant of the Jacobian matrix of the function $f$ at the point $x$.  We define the set $S_{f|_{_V}}$
  of all critical points of $f$ in $V$, that is, $S_{f|_{_V}} = \{ x \in V \; ; \; \det(Df(x)) = 0 \}$  and
  therefore $R_{f_{|_V}} = \re^n \setminus f(S_{f|_{_V}})$.

\begin{prop}
    Let $f \in \mathcal{C}^1(\overline{V}, \re^n)$ and consider  $b \in R_{f_{|_V}} \setminus f(\partial V)$  then
     the set $ f ^ {- 1} (\{b \}) $ is finite.

    \begin{proof}
        We have  that $ f $ is continuous and the unitary set $ \{b \} $ is closed, thus $ f ^ {- 1} (\{b \}) $
        is also a closed set on $\overline{V} $, consequently it is a closed set in $\re ^n$. It is also
        a limited set, because $f^{-1}(\{b \}) \subset V $  and $V$ is a limited set. Therefore $ f ^ {- 1} (\{b \}) $
         is closed and limited in $ \re ^ n $, i.e. it is a compact set.

        If $x \in f^{-1}(\{b \}) $ we have $\det(Df(x)) \neq 0$, then by  Inverse Map Theorem  $f$ is a
         diffeomorphism from a neighborhood $U_x$ of $x$ onto a neighborhood $\widehat{U}$ of $b$.
         Observe that $x$ is the unique point in $U_x$ such that $f(x)=b$. We have
        $$ f^{-1}(\{b \}) \subset \bigcup_{x \in f^{-1}(\{b \}) }  U_x.$$
        Using that $ f ^ {- 1} (\{b \}) $ is compact and $\{U_x \} $ is an open cover for $ f ^ {- 1} (\{b \}) $,
        then there is a finite set $\{x_1,\dots,x_k\}\subset  f^{-1}(\{b \})$ such that
        $$ f^{-1}(\{b \}) \subset \displaystyle\bigcup_{j=1 }^k  U_{x_j}.$$
        It implies that $ f ^ {- 1} (\{b \}) $ is finite because $U_{x_j}\cap f^{-1}(\{b \}) = \{x_j\}$.
    \end{proof}
\end{prop}

Now we are able to define the Brouwer topological degree.

\begin{defi}\label{defigraudebrouwer}
    Let $f \in \mathcal{C}^{\infty}(\overline{V}, \re^n)$ and consider $b \in R_{f_{|_V}} \setminus f(\partial V)$
     then we define the Brouwer topological degree of $ f $  relative to $ V $ at the point $ b $ as the integer
    \begin{equation*}
    d_B(f,V,b)=\displaystyle\sum_{x\in f^{-1}(\{b \})} {\sgn}  \det (Df(x)),
    \end{equation*}
    where $\sgn$  denotes the sign function and the set $f^{-1}(\{b \})$ is finite. If
    $f^{-1}(\{b \})=\emptyset$ then we define $d_B(f,V,b)=0$.
\end{defi}

\begin{obs}
    From the Definition \ref{defigraudebrouwer} we can see that $ d_B (f, V, b) = d_B (f-b, V, 0) $,
     because if we consider $ g = f-b $ we have $ g^{- 1} (\{ 0 \}) = f ^ {- 1} (\{ b \}) $.
\end{obs}

The Sard-Brown Theorem says that $R_{f_{|_V}}$ is dense $\re^n$,
which allow us to find regular values in any neighborhood of a
critical value.  Furthermore in \cite{OutRui2009} is proved that the
Brouwer topological degree is locally constant, that is, if $b \in
R_{f|_{_V}} \setminus f(\partial V)$ there is a neighborhood of $b$,
$W \subset R_{f|_{_V}} \setminus f(\partial V)$ such
 that $d_B(f,V,b)=d_B(f,V,a),$ for all $a \in W$.  Now we can define Brouwer topological degree for critical values or regular values.

\begin{defi}
    Let $f \in \mathcal{C}^{\infty}(\overline{V}, \re^n)$ and consider $b \in \re^n \setminus f(\partial V)$
    then we define the Brouwer topological degree of the $ f $  relative to $ V $ at the point $ b $ as the integer
    \begin{equation*}
    d_B(f,V,b)=d_B(f,V,a),
    \end{equation*}
    for any such regular value $a$ (which exist by the Sard-Brown Theorem).
\end{defi}

The next step is to define the Brouwer topological degree for
continuous functions. Consider $f \in \mathcal{C}(\overline{V}
,\re^n)$, $b \in \re^n\setminus f(\partial V)$ and $r =
\rho(b,f(\partial V))$ being the distance between point $b$ and set
$f(\partial V)$. According to Weirstrass Approximation Theorem we
find a polynomial (hence smooth) mapping $g$ such that  $||g-f|| <
\frac{r}{2}$. Next result ensures that we can define
$d_B(f,V,b)=d_B(g,V,b)$.

\begin{teo}
    Fix the set $U = \{ g \in \mathcal{C}^{\infty}(\overline{V} ,\re^n) \; ; \; ||g-f|| < \frac{r}{2} \}$
     then $d_B(g_1,V,b)=d_B(g_2,V,b)$, for $g_1, \; g_2 \in U$ and $b \in \re^n \setminus f(\partial V)$.
    \end{teo}
\begin{proof}
    For a proof and more details of this method we suggest \cite{OutRui2009}.
\end{proof}

Now we can define the Brouwer topological degree for continuous
functions.

\begin{defi}
The Brouwer degree for $f \in \mathcal{C}(\overline{V},\re^n)$ is
$d_B(f,V,b)=d_B(g,V,b)$,
     for  $g \in U$ and  $b \in \re^n \setminus f(\partial V)$.
\end{defi}

This final result straightforward  clarifies that the Brouwer degree
for simple zeroes of $\mathcal{C}^1$-functions is non-zero.

\begin{lema}
    Consider $ f \in \mathcal{C}^1(\Omega, \re ^ n )$, where $ \Omega$ is an open set of $\re^n$.
     If there is $ a \in  \Omega $ with $ f (a) = 0 $ and $ \det(Df (a)) \neq 0 $, then there is a
      neighborhood $ V $ of $ a $ such that $ f (x) \neq 0 $ for every $ x \in \overline {V} \setminus \{a \} $ and $ d_B (f, V, 0) \neq 0 $.
\end{lema}

\section*{Acknowlegements}

The authors thank Adriana Buic\u{a} for her comments in several
parts of this paper.

This work has received funding from the Ministerio de Econom\'{\i}a,
Industria y Competitividad - Agencia Estatal de Investigaci\'{o}n
(MTM2016-77278-P FEDER grant), the Ag\`{e}ncia de Gesti\'{o} d'Ajuts
Universitaris i de Recerca (2017 SGR 1617 grant), CAPES grant
88881.068462/2014-01, CNPq grant 304798/2019-3, and S\~{A}{\pounds}o Paulo Paulo Research Foundation (FAPESP) grants 2019/10269-3, 2018/05098-2 and 2016/00242-2.


\end{document}